\documentclass[a4paper,12pt,reqno]{amsart}
\usepackage{fullpage}
\usepackage{amssymb}
\usepackage{amsmath}
\usepackage{mathtools}
\usepackage{enumitem}
\usepackage{mathrsfs}
\usepackage[all]{xy}
\usepackage{mathtools}
\usepackage{hyperref}
\usepackage{url}
\usepackage{bbm}
\usepackage{longtable}

\usepackage{lmodern}
\usepackage[OT2, T1]{fontenc}

\DeclareSymbolFont{cyrletters}{OT2}{wncyr}{m}{n}
\usepackage[british]{babel}

\usepackage[margin=1.3in]{geometry}
\numberwithin{equation}{section} \numberwithin{figure}{section}

\makeatletter
\@namedef{subjclassname@2020}{\textup{2020} Mathematics Subject Classification}
\makeatother

\DeclareMathOperator{\Gal}{Gal} 
\DeclareMathOperator{\Ker}{Ker}
\DeclareMathOperator{\Aut}{Aut}

\DeclareMathOperator{\Hom}{Hom} \DeclareMathOperator{\re}{Re}
\DeclareMathOperator{\im}{im}   
\DeclareMathOperator{\vol}{vol}

\DeclareMathOperator{\ord}{ord}

\DeclareMathOperator{\Res}{Res} \DeclareMathOperator{\Norm}{N}

\DeclareMathOperator{\Frob}{Frob}
\DeclareMathOperator{\HH}{H}
\DeclareMathOperator{\id}{id}

\let\Im\relax
\DeclareMathOperator{\Im}{Im} 
\DeclareMathOperator{\Cl}{Cl} 

\DeclareMathSymbol{\Sha}{\mathalpha}{cyrletters}{"58}

\renewcommand{\O}{\mathcal{O}}
\newcommand{\OO}{\mathcal{O}}

\newcommand{\dual}[1]{{#1}^{\wedge}}

\newcommand{\one}{\mathbbm{1}}

\newcommand\F{\mathbb{F}}

\newcommand\Z{\mathbb{Z}}
\newcommand\ZZ{\mathbb{Z}}
\newcommand\N{\mathbb{N}}
\newcommand\Q{\mathbb{Q}}
\newcommand\R{\mathbb{R}}
\newcommand\C{\mathbb{C}}
\newcommand\CC{\mathbb{C}}

\newcommand{\Adele}{\mathbf{A}}

\newcommand{\bbQ}{{\mathbb Q}}

\newcommand{\bbZ}{{\mathbb Z}}

\newcommand{\cA}{{\mathcal A}}

\newcommand{\cO}{{\mathcal O}}

\newcommand{\fp}{\mathfrak{p}}
\newcommand{\fq}{\mathfrak{q}}
\newcommand{\gextk}{\ensuremath{G\text{-ext}(k)}}

\newcommand{\Val}{\Omega}

\newtheorem{lemma}{Lemma}

\newtheorem{theorem}[lemma]{Theorem}
\newtheorem{proposition}[lemma]{Proposition}

\theoremstyle{definition}
\newtheorem{example}[lemma]{Example}

\newtheorem{definition}[lemma]{Definition}
\newtheorem{remark}[lemma]{Remark}

\newtheorem*{ack}{Acknowledgements}

\usepackage[usenames,dvipsnames]{color}

\numberwithin{lemma}{section}

\begin{document}

\title{Distribution of genus numbers of abelian number fields}

\author{\sc Christopher Frei}
\address{Christopher Frei\\
TU Graz\\
Institute of Analysis and Number Theory\\
Steyrergasse 30/II\\
8010 Graz\\
Austria.}
\email{frei@math.tugraz.at}
\urladdr{https://www.math.tugraz.at/~frei/}

\author{Daniel Loughran}
  \address{Daniel Loughran \\
	Department of Mathematical Sciences\\
	University of Bath\\
	Claverton Down\\
	Bath\\
	BA2 7AY\\
	UK}
\urladdr{https://sites.google.com/site/danielloughran/}

\author{\sc Rachel Newton}
\address{Rachel Newton\\
Department of Mathematics\\ 
King's College London\\
Strand\\ 
London\\
WC2R 2LS\\
UK.}
   \email{rachel.newton@kcl.ac.uk}
\urladdr{https://racheldominica.wordpress.com/}


\begin{abstract}
We study the quantitative behaviour of genus numbers of abelian extensions of number fields with given Galois group. We prove an asymptotic formula for the average value of the genus number and show that any given genus number appears only $0\%$ of the time.
\end{abstract}

\maketitle

\thispagestyle{empty}

\tableofcontents

\section{Introduction}

Let $k$ be a number field. The Cohen--Lenstra heuristics give a prediction for
the distribution of the odd part of the class groups of quadratic extensions of $k$ \cite{CL84}. These are  known for the $3$-torsion in the class group \cite{DH71, DW88}, 
but are wide open in general. 
On the other hand, the $2$-torsion in the class group of a
quadratic extension admits a very simple description via Gauss's genus
theory. It is this part of the class group which we shall focus on in this paper,
in the more general setting of genus numbers of abelian extensions. There has
been a recent interest in statistics regarding the genus group for other classes of field
extensions, for example for cubic and quintic extensions \cite{MT16, MTT20},
with the only previous cases considered in the abelian setting being cyclic extensions
of $\Q$ of prime degree \cite{Kim20}. The definition of the genus group for an abelian extension
is as follows.

\begin{definition} \label{def:genus}
	Let $K/k$ be an abelian extension.
	The \emph{genus field} of $K/k$ is the largest
	extension $\mathfrak{G}_{K/k}$ of $K$ which is unramified at all places of $K$
	and such that $\mathfrak{G}_{K/k}$ is an abelian extension of $k$.
	The \emph{genus group} of $K/k$  is the Galois group $\Gal(\mathfrak{G}_{K/k} /K)$.
	The \emph{genus number} $\mathfrak{g}_{K/k}$ of $K/k$ is the size of the genus group.
      \end{definition}

At archimedean places, we use the convention that $\C/\R$ is  ramified.

By class field theory, there is a tower of fields $K \subset\mathfrak{G}_{K/k}
\subset H_K$. Here $H_K$ is the Hilbert class field of $K$ and the class group
$\Cl(K)$ of $K$ is canonically isomorphic to $\Gal(H_K/K)$. The subgroup
$\Gal(H_K/\mathfrak{G}_{K/k}) \subset \Cl(K)$ is called the \emph{principal
  genus} of $K/k$  and the genus group is the quotient
$\Cl(K)/\Gal(H_K/\mathfrak{G}_{K/k})$ of the class group. Our main theorem
concerns the average size of the genus number as one sums over all abelian
extensions of bounded conductor with given Galois group. Write $\Phi(K/k)$ for
the absolute norm of the conductor of $K$ (viewed as an ideal in $k$).
\begin{theorem} \label{thm:genus}
	Let $G$ be a non-trivial finite abelian group. Then
	$$\sum_{\substack{\Gal(K/k) \cong G \\ \Phi(K/k) \leq B}} \mathfrak{g}_{K/k} \sim c B (\log B)^{\varrho(k,G)-1},$$
	for some positive constant $c$, where
        \begin{equation}\label{eq:def_varrho_k_G}
	\varrho(k,G) = \sum_{g\in G\smallsetminus\{\id_G\}}\frac{\ord g}{[k(\mu_{\ord g}): k]}.        
      \end{equation}
\end{theorem}
Our main result (Theorem \ref{thm:genus_constant}) also gives an explicit expression for the leading constant in the asymptotic formula.
To prove our results we consider the Dirichlet series
$$\sum_{\substack{\Gal(K/k) \cong G \\ \Phi(K/k) \leq B}} \frac{\mathfrak{g}_{K/k}}{\Phi(K/k)^s}.$$
Using class field theory, we may rewrite this in terms of an
idelic series, since the genus number has an idelic interpretation (Lemma
\ref{lem:genus_formula}). We then use Poisson summation to express this series
in terms of the Dedekind zeta function of $k$.

\begin{remark}\label{rmk:average}
	It is illustrative to compare the asymptotic in Theorem \ref{thm:genus} with the unweighted
	count of $G$-extensions, first established in \cite{Woo10}:
	$$\sum_{\substack{\Gal(K/k) \cong G \\ \Phi(K/k) \leq B}} 1 \sim c' B (\log B)^{\omega(k,G)-1},\quad
	\omega(k,G) = \sum_{g\in G\smallsetminus\{\id_G\}}\frac{1}{[k(\mu_{\ord g}): k]}.        
	$$
	In particular, one can interpret Theorem \ref{thm:genus} as saying that the 
	average value of the genus number is
	$$\frac{c}{c'}(\log \Phi(K/k))^{\sum_{g\in G\smallsetminus\{\id_G\}}(\ord g -1)/[k(\mu_{\ord g}): k]}. $$
	Note that this behaviour is in stark contrast to the case of cubic and quintic extensions \cite{MT16, MTT20} of $\Q$ (ordered by discriminant), where the average value of the genus number is constant.
\end{remark}

\begin{example} 
	\hfill
	\begin{enumerate}
	\item 
	For $\ell$ prime we have
	$$\varrho(k,\Z/\ell\Z) = \frac{\ell(\ell-1)}{[k(\mu_{\ell}):k]}.
	$$
	In the case $k = \Q$, this recovers \cite[Thm.~4.2]{Kim20} (here
	$\varrho(\Q,\Z/\ell\Z) = \ell$). All other cases are new.
	The following examples
	are all completely new.
	\item Biquadratic extensions:
	$$\varrho(k,(\Z/2\Z)^2) =
		6.
	$$
	\item A cyclic extension of non-prime degree:
		$$\varrho(k,\Z/4\Z) =
	\begin{cases}
		6, & \mu_4 \not \subset k \\
		10, & \mu_4 \subset k.	
	\end{cases}
	$$
	\end{enumerate}
\end{example}

Remark \ref{rmk:average} and the  case of quadratic fields suggest
that there should be very few $G$-extensions with any fixed genus number
$g$. Our next theorem confirms this by showing that every genus number occurs
at most $0\%$ of the time. Again, this is in stark contrast to the results for
cubic and quintic extensions ordered by discriminant~\cite{MT16, MTT20}.

\begin{theorem}\label{thm:zero_percent}
  Let $k$ be a number field, $G$ a non-trivial finite abelian group and $g\in\N$. Then
  \begin{equation*}
    \lim_{B\to\infty}\frac{|\{K/k\ :\ \Gal(K/k)\cong G,\ \Phi(K/k)\leq
      B\ \text{ and }\ \mathfrak{g}_{K/k}=g\}|}{|\{K/k\ :\ \Gal(K/k)\cong G,\ \Phi(K/k)\leq
      B\}|}=0.
  \end{equation*}
\end{theorem}

\begin{ack}
We thank Alex Bartel for useful discussions on genus groups and Hendrik Lenstra for asking the question that led to Theorem~\ref{thm:zero_percent}. We thank the anonymous referee for their careful reading of our paper.
Christopher Frei was supported by EPSRC Grant EP/T01170X/1 and
EP/T01170X/2. Daniel Loughran was supported by UKRI Future Leaders
Fellowship MR/V021362/1. Rachel Newton was supported by EPSRC Grant
EP/S004696/1 and EP/S004696/2, and UKRI Future Leaders Fellowship
MR/T041609/1 and MR/T041609/2.
\end{ack}

\section{Harmonic analysis}
To prove Theorem \ref{thm:genus}, we will use a formula for the genus number (Lemma~\ref{lem:genus_formula}) which expresses $\mathfrak{g}_{K/k}$ in terms of various invariants of the extension $K/k$. Most importantly for us, it reduces our task to counting abelian extensions for which a given unit is everywhere locally a norm, weighted by the products of the ramification indices. We will achieve this via a general result proved using harmonic analysis, which we state as Theorem~\ref{thm:local_norms_weighted_by_ramification} below. Let $G$ be a finite abelian group, let $F$ be a field and $\bar{F}$ a separable closure of $F$. We define a \emph{sub-$G$-extension} of $F$ to be a continuous homomorphism $\Gal(\bar{F}/F)\to G$. A sub-$G$-extension corresponds to a pair $(L/F,\psi)$, where $L/F$ is a Galois extension inside $\bar{F}$ and $\psi$ is an injective homomorphism $\Gal(L/F)\to G$. A \emph{$G$-extension} of $F$ is a surjective continuous homomorphism $\Gal(\bar{F}/F)\to G$. 

Having fixed an algebraic closure $\bar{k}$ of the number field $k$, we write $\gextk$ for the set of $G$-extensions of $k$.
For each place $v$ of $k$, we fix an algebraic closure $\bar{k}_v$ and compatible embeddings $k\hookrightarrow\bar{k}\hookrightarrow\bar{k}_v$ and $k\hookrightarrow k_v \hookrightarrow \bar{k}_v$. 
Hence, a sub-$G$-extension $\varphi$ of $k$ induces a sub-$G$-extension $\varphi_v$ of $k_v$ at every place $v$.
We write $\mathfrak{e}_v(\varphi)=\mathfrak{e}(\varphi_v)=\mathfrak{e}(L_v/k_v)$ for the ramification index of the extension $L_v/k_v$ given by $\varphi_v$. Our main counting result is the following:

\begin{theorem}\label{thm:local_norms_weighted_by_ramification}
	Let $k$ be a number field, $G$ a non-trivial finite abelian group and $\mathcal{A} \subset  k^\times$ a finitely generated subgroup.
	Let $S$ be a finite set of places of $k$ and for $v \in S$ let 
	$\Lambda_v$ be a set of sub-$G$-extensions of $k_v$.
	For $v \notin S$ let $\Lambda_v$ be the set of sub-$G$-extensions of $k_v$ corresponding
	to those extensions of local fields $L/k_v$ for which every element of $\mathcal{A}$ is a local norm from $L/k_v$.
	Let $\Lambda := (\Lambda_v)_{v \in \Omega_k}$. Then there exist
	$ c_{k,G,\Lambda} \geq 0$ and $\delta=\delta(k,G,\mathcal{A}) >0$
	such that
	\begin{equation*}
		\sum_{\substack{\varphi\in\gextk\\\Phi(\varphi)\leq B\\\varphi_v\in\Lambda_v \forall v\in\Omega_k}}\prod_v\mathfrak{e}_v(\varphi) = c_{k,G,\Lambda}B(\log B)^{\varrho(k,G,\mathcal{A})-1}
		+  O(B(\log B)^{\varrho(k,G,\mathcal{A})-1-\delta}), \quad B\to\infty,
	\end{equation*}
	where
        \begin{equation*}
          \varrho(k,G,\mathcal{A}) = \sum_{g\in G\smallsetminus\{\id_G\}}\frac{\ord g}{[k_{\ord g}: k]}
          \quad \text{and }\quad  k_d=k(\mu_d,\sqrt[d]{\cA}).
        \end{equation*}
	Moreover $c_{k,G,\Lambda} >0$ if there exists a 
	sub-$G$-extension of $k$
	which realises the given local conditions for all places $v$.
\end{theorem}

Theorem \ref{thm:local_norms_weighted_by_ramification} has a very similar statement to \cite[Thm.~3.1]{HNP2}; the primary difference is that the sum in Theorem~\ref{thm:local_norms_weighted_by_ramification} is weighted by the product of ramification indices $\prod_v\mathfrak{e}_v(\varphi)$. We prove Theorem \ref{thm:local_norms_weighted_by_ramification} using the method, based upon harmonic analysis, that we developed to prove \cite[Thm.~3.1]{HNP2}, and many of the steps are formally similar. The  ramification indices come into play when calculating the relevant local Fourier transforms, and these in turn change the singularity type of the resulting Dirichlet series. We now begin the proof of Theorem \ref{thm:local_norms_weighted_by_ramification}.

\subsection{M\"{o}bius inversion and Poisson summation} \label{sec:Mob}
These steps are very similar to those taken in~\cite[\S 3]{HNP} and~\cite[\S3]{HNP2}, so we shall be brief. To prove the result we are free to increase the set $S$, so we assume that $S$ contains all archimedean places of $k$ and all places of $k$ lying above the primes $p \leq |G|^2$. Moreover, we assume that $\mathcal{A} \subset \O_S^\times$ and that $\O_S$ has trivial class group.
Let 
$$f_{\Lambda_v}: \Hom(\Gal(\bar{k}_v/k_v),G) \to \Z, \quad \varphi_v \mapsto \one_{\Lambda_v}(\varphi_v) \mathfrak{e}_v(\varphi_v)$$
where $\mathfrak{e}_v(\varphi_v)$ denotes the ramification index of $\varphi_v$ and $\one_{\Lambda_v}$ the indicator function of $\Lambda_v$. We let $f_\Lambda$ be the product of the $f_{\Lambda_v}$. The Dirichlet series for our counting problem is
\begin{equation}\label{eq:def_F_s}
F(s) = \sum_{\varphi \in \gextk}\frac{f_\Lambda(\varphi)}{\Phi(\varphi)^s}.
\end{equation}
We perform M\"{o}bius inversion to write this as
$$
F(s) = \sum_{H \subset G}\mu(G/H)\sum_{\varphi \in \Hom(\Gal(\bar{k}/k),H)}\frac{f_\Lambda(\varphi)}{\Phi(\varphi)^s}
$$
where the sum is over subgroups $H$ of $G$ and $\mu$ denotes the M\"{o}bius function on isomorphism classes of finite abelian groups (cf.~\cite[\S3.3.1]{HNP2}). By global class field theory, we have the identification
$$\Hom(\Gal(\bar{k}/k),H) = \Hom(\Adele_k^\times/k^\times,H),$$
where $\Adele_k^\times$ denotes the id\`eles of $k$.
This allows us to view $f_\Lambda$ as a function on $\Hom(\Adele_k^\times/k^\times,H)$, and leads to the expression
\begin{equation} \label{eqn:Mob}
F(s) = \sum_{H \subset G}\mu(G/H)\sum_{\chi \in \Hom(\Adele_k^\times/k^\times,H)}\frac{f_\Lambda(\chi)}{\Phi(\chi)^s}.
\end{equation}
We approach the inner sums via Poisson summation. 
For each place $v$, we equip the finite group $\Hom(k^\times_v,H)$ with the unique Haar measure $\mathrm{d}\chi_v$ such that 
$$\vol(\Hom(k_v^\times/\OO_v^\times,H)) = 1$$
(for $v$ archimedean, we take $\OO_v = k_v$ by convention).
The product of these measures yields a well-defined measure $\mathrm{d}\chi$ on $\Hom(\Adele_k^\times,H)$. We say that an element of $\Hom(k^\times_v,H)$ is \emph{unramified} if it lies in the subgroup $\Hom(k_v^\times/\OO_v^\times,H)$, i.e.\ if it is trivial on $\OO_v^\times$.
The Pontryagin dual of $\Hom(\Adele_k^\times,H)$ is naturally identified with $\Adele_k^\times \otimes \dual{H}$, where $\dual{H}=\Hom(H,\C^\times)$ denotes the Pontryagin dual of $H$ (similarly with $\Adele_k^\times$ replaced by $\Adele_k^\times/k^\times$ or $k_v^\times$; cf.~\cite[\S3.1]{HNP}).

The function $f_\Lambda/\Phi^s$ is a product of local functions $f_{\Lambda_v}/\Phi_v^s$ on $\Hom(k_v^\times,H)$, where $\Phi_v(\chi_v)$ is the reciprocal of the $v$-adic norm of the conductor of $\Ker\chi_v$. For $v\notin S$, these local functions take only the value $1$ on the unramified elements by our choice of $S$, thus $f_\Lambda/\Phi^s$ extends to a well-defined continuous function on $\Hom(\Adele_k^\times,H)$. We define its Fourier transform to be
$$\widehat{f}_{\Lambda,H}(x;s) = \int_{\chi \in \Hom(\Adele_k^\times,H)}
\frac{f_\Lambda(\chi) \langle \chi, x \rangle}{\Phi(\chi)^s} \mathrm{d}\chi,$$
where $x = (x_v)_v \in  \Adele_k^\times \otimes H^\wedge$.  Similarly, for $x_v \in k_v^{\times} \otimes \dual{H}$ we have the local Fourier transform
$$\widehat{f}_{\Lambda_v,H}(x_v;s) = \int_{\chi_v \in \Hom(k_v^\times,H)}
\frac{f_{\Lambda_v}(\chi_v) \langle \chi_v, x_v \rangle}{\Phi_v(\chi_v)^s} \mathrm{d}\chi_v.$$
For $\re s \gg 1$, the global Fourier transform exists and defines a holomorphic function in this domain, and there is an Euler product decomposition
\begin{equation} \label{eqn:Euler_product}
	\widehat{f}_{\Lambda,H}(x;s) = \prod_v \widehat{f}_{\Lambda_v,H}(x_v;s).
\end{equation}
As in \cite[Prop.~3.9]{HNP2}, applying Poisson summation one obtains the following:
\begin{equation}\label{eq:poisson}
	 \sum_{\chi \in \Hom(\Adele_k^\times/k^{\times} , H)} \frac{f_\Lambda(\chi)}{\Phi(\chi)^{s}}
	=\frac{1}{|\OO_k^\times\otimes\dual{H}|} 
	\sum_{x \in \OO_S^\times \otimes \dual{H}}\widehat{f}_{\Lambda, H}(x;s), \quad \re s > 1.
\end{equation}

To analyse the Poisson sum, we need to calculate the Fourier transforms. We begin by studying the local Fourier transforms.

\subsection{Local Fourier transforms}
We first give a formula for the ramification index via local class field theory.
For any place $v$ of $k$ we have the exact sequence
$$1	\to \O_v^\times \to k_v^\times \to k_v^\times/\O_v^\times \to 1.$$
A choice of uniformiser gives an isomorphism
$k_v^\times \cong \O_v^\times \bigoplus k_v^\times/\O_v^\times$
and thus
$$\Hom(k_v^\times, H) \cong \Hom(\O_v^\times, H) \bigoplus 
\Hom(k_v^\times/\O_v^\times, H).$$
For a character $\chi_v: k_v^\times \to H$, we write its decomposition as 
$$\chi_v = (\chi_{v,\textrm{r}}, \chi_{v,\textrm{nr}})$$
($\chi_{v,\textrm{r}}$ is the ``ramified part'' and $\chi_{v,\textrm{nr}}$ is the ``non-ramified part''). The following lemma is standard class field theory; we include its proof for completeness. 

\begin{lemma} \label{lem:ram_char}
	We have $\mathfrak{e}_v(\chi_v) =|\chi_v(\OO_v^\times)|= \lvert\im \chi_{v,\textrm{r}}\rvert$. If $\chi_v$ is tamely ramified then $\im \chi_{v,\textrm{r}}$ is a cyclic group and therefore $\mathfrak{e}_v(\chi_v)= \ord( \chi_{v,\textrm{r}})$.
\end{lemma}
\begin{proof}
The result for $v$ archimedean is immediate so let us assume that $v$ is non-archimedean. Let $K_w$ denote the field extension of $k_v$ corresponding via class field theory to $\chi_v$, and write $\O_{w}$ for the ring of integers of $K_{w}$. Let $\ord_v: k_v^\times\twoheadrightarrow\Z$ be the normalised valuation on $k_v$ and let $\ord_w: K_w^\times\twoheadrightarrow\Z$ be the normalised valuation on $K_w$. Since $\ord_v=\frac{1}{\mathfrak{e}_v(\chi_v)}\ord_w$ and $[K_w:k_v]=\mathfrak{e}_v(\chi_v)\mathfrak{f}_v(\chi_v)$, where $\mathfrak{f}_v$ denotes the residue degree, the following diagram commutes:
\[
\xymatrix{
0\ar[r] & \O_{w}^\times\ar[r]\ar[d]^{N_{K_w/k_v}}& K_w^\times\ar[r]^{\ord_w}\ar[d]^{N_{K_w/k_v}}& \Z\ar[r] \ar[d]^{\mathfrak{f}_v(\chi_v)}&0\\
 0\ar[r] & \O_{v}^\times\ar[r]& k_v^\times\ar[r]^{\ord_v}& \Z\ar[r] &0.
}
\]
The snake lemma gives an exact sequence
\[0 \to \O_{v}^\times/N_{K_w/k_v}(\O_{w}^\times)\to k_v^\times/N_{K_w/k_v}(K_w^\times)\to \Z/\mathfrak{f}_v(K)\Z \to 0.\]
The local reciprocity map gives an isomorphism $k_v^\times/N_{K_w/k_v}(K_w^\times)\to\Gal(K_w/k_v)$, whereby $|k_v^\times/N_{K_w/k_v}(K_w^\times)|=[K_w:k_v]=\mathfrak{e}_v(\chi_v)\mathfrak{f}_v(\chi_v)$. Therefore, \[\mathfrak{e}_v(\chi_v)=|\O_{v}^\times/N_{K_w/k_v}(\O_{w}^\times)|=|\O_{v}^\times/\Ker\chi_{v,\textrm{r}}|,\]
as required. Now let $\mathfrak{m}_v$ denote the maximal ideal of $\O_v$.
If $K_w/k_v$ is tamely ramified then $1+\mathfrak{m}_v\subset N_{K_w/k_v}(\O_{w}^\times)=\Ker\chi_{v,\textrm{r}}$ and hence the cyclic group $\F_v^\times$ surjects onto $\im \chi_{v,\textrm{r}}$.
\end{proof}

\begin{lemma}\label{lem:unram_inv_norms}
Let $v\notin S$. The function $f_{\Lambda_v}$ is $\Hom(k_v^\times/\cO_v^\times, H)$-invariant and $f_{\Lambda_v}(\psi_v)=1$ for all $\psi_v\in \Hom(k_v^\times/\cO_v^\times, H)$.
\end{lemma}

\begin{proof}
This follows from Lemma~\ref{lem:ram_char} and \cite[Lemma~3.7]{HNP2}.
\end{proof}

For $x\in k^\times\otimes\dual H$, we denote by $x_v$ the
image of $x$ under $k^\times\otimes\dual H \to k_v^\times\otimes\dual
H$. By $\cA_v$ we denote the image of our finitely generated
subgroup $\cA\subset k^{\times}$ under  $k^\times\subset k_v^{\times}$. Recall that, by
our choice of $S$, we have $\cA_v\subset\OO_v^\times$ for all $v\notin S$.

\begin{lemma}\label{lem:fourier_transform}
  For $v\notin S$ and
  $x\in\OO_S^\times\otimes \dual H$, 
  let
  \begin{equation*}
    s_{x,H}(v) = -1+\sum_{\chi_v\in\Hom(\OO_v^\times/\cA_v,H)}\ord(\chi_v)\langle \chi_v,x_v\rangle.
  \end{equation*}
  Then
  \begin{equation*}
    \widehat{f}_{\Lambda_v,H}(x_v;s) = 1+\frac{s_{x,H}(v)}{q_v^s}.
  \end{equation*}
\end{lemma}

\begin{proof}
Using Lemma~\ref{lem:unram_inv_norms} and following the proof of \cite[Lemma~3.8]{HNP2} gives
\[ \widehat{f}_{\Lambda_v,H}(x_v;s) =\sum_{\chi_v\in\Hom(\OO_v^\times,
    H)}\frac{f_{\Lambda_v}(\chi_v)\langle \chi_v,x_v\rangle}{\Phi(\chi_v)^s}.\]
Now mimic the start of the proof of \cite[Lemma~3.10]{HNP2} and apply
Lemma~\ref{lem:ram_char} to obtain the result.
\end{proof}

\subsection{Frobenian functions} \label{sec:frob}
We will analyse the global Fourier transforms using the theory of frobenian functions from Serre's book \cite[\S 3.3]{Ser12}. The parts of the theory relevant for us are also summarised in \cite[\S2]{HNP2}.
Recall that a \emph{class function} on a group is a function which is constant on conjugacy classes.

\begin{definition} \label{def:frob}
	Let $k$ be a number field and 
	$\rho: \Val_k \to \CC$ a function on the set of places of $k$.
	Let $S$ be a finite set of places of $k$.
	We say that $\rho$ is $S$-\emph{frobenian} if there exist 
	\begin{enumerate}
		\item[(a)] a finite Galois extension $K/k$, with Galois group $\Gamma$, such that $S$ contains all places which ramify in $K/k$, and
		\item[(b)] a class function $\varphi: \Gamma \to \CC$,
	\end{enumerate}
	such that for all $v \not \in S$ we have
	$$\rho(v) = \varphi(\Frob_v),$$
	where $\Frob_v \in \Gamma$ denotes the Frobenius element of $v$.
	We say that $\rho$ is
        \emph{frobenian} if it is $S$-frobenian
	for some $S$.
	A subset of $\Val_k$ is called ($S$-)\emph{frobenian} if its indicator function is ($S$-)frobenian.	
\end{definition}

In Definition \ref{def:frob}, we adopt a common abuse of notation (see \cite[\S3.2.1]{Ser12}), and denote by $\Frob_v \in \Gamma$ the choice of some element of the Frobenius conjugacy class at $v$; note that $\varphi(\Frob_v)$ is well defined as $\varphi$ is a class function.
We define the \emph{mean} of $\rho$ to be 
$$m(\rho) = \frac{1}{|\Gamma|} \sum_{\gamma \in \Gamma}\varphi(\gamma)\in\CC.$$ 

\subsection{Global Fourier transforms}
Our next aim is to show that the function $s_{x,H}$ from Lemma \ref{lem:fourier_transform} is frobenian. This will allow us to obtain analytic continuations of the corresponding global Fourier transforms (possibly with branch point singularities).

For $x_v\in\OO_v^\times\otimes\dual H$, we abuse notation slightly and write $x_v\in\OO_v^{\times d}\otimes\dual H$ if $x_v$ is in the image of the not necessarily injective map $\OO_v^{\times d}\otimes \dual H \to \OO_v^\times\otimes \dual H$. 
 
 \begin{lemma}\label{lem:dhv_frobenian}
   Let $x\in \OO_S^\times\otimes \dual H$ and let $e$ be the exponent of
   $H$. For $v\notin S$, let 
   \begin{align*}
	d_{x,H}(v)&=\max\{d \mid \gcd(e,q_v-1) :  x_v \in \OO_v^{\times d}\otimes \dual{H}\},     \\
	d_{\cA,H}(v)&=\max\{d\mid \gcd(e,q_v-1): \cA\bmod v \subseteq \F_v^{\times d}\}.
   \end{align*}
   Then any function $\Val_k\to\CC$ whose restriction to
   $\Val_k\smallsetminus S$ is either $d_{x,H}$ or $d_{\cA,H}$ is $S$-frobenian.
 \end{lemma}
 \begin{proof}
   We choose a presentation $\dual H = \ZZ/n_1\ZZ \oplus\cdots\oplus\ZZ/n_l\ZZ$, thus identifying $x\in\OO_S^\times\otimes\dual H$ with a tuple $(x_1\OO_S^{\times n_1},\ldots,x_l\OO_S^{\times n_l})\in \OO_S^\times/\OO_S^{\times n_1}\oplus\cdots\oplus\OO_S^\times/\OO_S^{\times n_l}$.  
   Then $x_v\in \OO_v^{\times d}\otimes\dual H$ if and only if $x_i\in\OO_v^{\times (d,n_i)}$ for all $1\leq i\leq l$.

   For all $d\mid e$, let $K_{d} = k(\mu_d,x_1^{1/\gcd(d,n_1)},\ldots,x_l^{1/\gcd(d,n_l)})$, and define the sets
   \begin{equation*}
     \Upsilon_d = \Gal(K_e/K_d)\smallsetminus \bigcup_{\substack{d'\mid \frac{e}{d}\\d'\neq 1}}\Gal(K_e/K_{dd'})\subset \Gal(K_e/k).
   \end{equation*}
   The $\Upsilon_d$ form a partition of $\Gal(K_e/k)$, and each $\Upsilon_d$ is conjugacy invariant since the involved subgroups are normal. A place $v\notin S$ satisfies $\Frob_v\in\Upsilon_d$ if and only if $d$ is the largest divisor of $e$ such that $v$ splits completely in $K_d$. By the observations made at the start of this proof,  $v$ splits completely in $K_d$ if and only if $d\mid q_v-1$ and $x_v\in\OO_v^{\times d}\otimes\dual H$.
   
   Hence, if $\varphi:\Gal(K_e/k)\to\CC$ is the class function that takes the value $d$ on $\Upsilon_d$, then $d_{x,H}(v)=\varphi(\Frob_v)$ for all $v\notin S$.

   That $d_{\cA,H}$ is $S$-frobenian is proved in \cite[Lemma~3.11]{HNP2}. The proof is similar to the above and we briefly recall the relevant construction as we shall use it later.  For every $d\mid e$, we define the number field $k_d=k(\mu_d,\sqrt[d]{\mathcal{A}})$. The subsets
  \begin{equation*}
    \Sigma_d = \Gal(k_e/k_d)\smallsetminus\bigcup_{\substack{d'\mid \frac{e}{d}\\d'\neq 1}}\Gal(k_e/k_{dd'}) \subset \Gal(k_e/k)
  \end{equation*}
  are conjugacy invariant and form a partition of $\Gal(k_e/k)$. Let $\varphi : \Gal(k_e/k)\to\CC$ be the class function that takes the constant value $d$ on $\Sigma_d$, for all $d\mid e$. Then $d_{\mathcal{A},H}(v)=\varphi(\Frob_v)$ for all $v\notin S$, so in particular it is $S$-frobenian.
 \end{proof}
 
\begin{proposition}\label{prop:s_S-frob}
  Let $x\in\OO_S^\times\otimes \dual H$. Then any function $\Val_k\to\CC$ that
  sends $v\notin S$ to $s_{x,H}(v)$ is $S$-frobenian. Moreover,
  $s_{x,H}(v)\in\R$ for all $v\notin S$.
\end{proposition}

\begin{proof}
  For $v\notin S$, reduction modulo $v$ yields an isomorphism 
  $\Hom(\OO_v^\times,H)\cong \Hom(\F_v^\times,H)$. If $m\mid q_v-1$ is
  maximal with $\cA\bmod v\subseteq\F_v^{\times m}$, then \mbox{$\cA\bmod
  v=\F_v^{\times m}$}. This shows that any $\chi_v: \OO_v^\times/\cA_v \to H$
  has order dividing $d_{\cA,H}(v)$, and for $m\mid d_{\cA,H}(v)$, the set
  $\{\chi_v: \OO_v^\times/\cA_v \to H\ :\ \chi_v \textrm{ has order }
  m\}$ can be naturally identified with
  $\{\chi_v: \OO_v^\times\to H\ :\ \Ker \chi_v=\OO_v^{\times m}\}$. Now
  M\"{o}bius inversion yields \begin{align*} 1+s_{x,H}(v)&=\sum_{m\mid
      d_{\cA,H}(v)}m\sum_{\substack{\chi_v\in\Hom(\OO_v^\times,H)\\\Ker\chi_v=\OO_v^{\times
          m}}}\langle \chi_v,x_v\rangle\\ &=\sum_{m\mid
      d_{\cA,H}(v)}m\sum_{d\mid
      m}\mu\left(\frac{m}{d}\right)\sum_{\chi_v\in\Hom(\OO_v^\times/\OO_v^{\times
        d},H)}\langle \chi_v,x_v\rangle.
  \end{align*}
  Character orthogonality shows that for all $d\mid \gcd(e,q_v-1)$ we have
\begin{equation*}
  \sum_{\chi_v\in\Hom(\OO_v^\times/\OO_v^{\times d},H)}\langle \chi_v,x_v\rangle=
  \begin{cases}
    \lvert\Hom(\OO_v^\times/\OO_v^{\times d},H)\rvert =|H[d]| &\text{ if }d\mid d_{x,H}(v),\\
    0 &\text{ otherwise,}
  \end{cases}
\end{equation*}
where $d_{x,H}(v)$ was defined in Lemma \ref{lem:dhv_frobenian}. Letting
\begin{equation}\label{eq:sigma_frobenian_form_2}
  F(A,B)=-1+\sum_{m\mid A}m\sum_{d\mid \gcd(m,B)}\mu\left(\frac{m}{d}\right)|H[d]|,
\end{equation}
we have shown that $s_{x,H}(v)=F(d_{\cA,H}(v),d_{x,H}(v))$. This is $S$-frobenian, as the functions $v\mapsto d_{\cA,H}(v)$ and $v\mapsto d_{x,H}(v)$ are $S$-frobenian by Lemma~\ref{lem:dhv_frobenian}. From \eqref{eq:sigma_frobenian_form_2} it is also clear that $s_{x,H}(v)\in\R$.
\end{proof}

Again, we abuse notation
and denote by $\cA_v\otimes\dual{H}$ the image of the not necessarily injective
map $\cA_v\otimes\dual{H}\to k_v^{\times}\otimes\dual{H}$. For $v\notin S$, we have $\cA_v\otimes\dual{H}\subseteq\OO_v^\times\otimes\dual{H}$.

\begin{proposition}
Let $\varrho(k,H,\mathcal{A})$ be defined as in
Theorem~\ref{thm:local_norms_weighted_by_ramification} (with $G$ replaced by
$H$) and let $\varrho(k,H,\mathcal{A}, x)\in\R$ denote the mean of an
$S$-frobenian function given by $s_{x,H}$ from Proposition~\ref{prop:s_S-frob}.
\begin{enumerate}
\item\label{leq} For all $x\in\OO_S^\times\otimes \dual H$, we have $\varrho(k,H,\mathcal{A}, x)\leq \varrho(k,H,\mathcal{A}, 1)$, with equality if and only if $x_v\in \cA_v\otimes\dual{H}\ \forall v\notin S$.
\item\label{mean} We have $\varrho(k,H,\mathcal{A}, 1)=\varrho(k,H,\mathcal{A}).$
\end{enumerate}
\end{proposition}

\begin{proof}
Part~\eqref{leq} is clear since, by definition of $s_{x,H}$ and character orthogonality, for all $v\notin
S$ we have $s_{x,H}(v)\leq s_{1,H}(v)$, with equality if and only if
$\langle\chi_v, x_v \rangle=1$ for all
$\chi_v\in\Hom(\OO_v^\times/\cA_v, H)$.
For part~\eqref{mean}, writing $k_d=k(\mu_d,\sqrt[d]{\cA})$, it suffices to
show that the mean of $s_{1,H}(v)+1$ equals $\sum_{h\in H}\frac{\ord
  h}{[k_{\ord h}:k]}$. By definition, for $v\notin S$ we have
\[s_{1,H}(v)+1=\sum_{\chi_v: \OO_v^\times/\cA_v \to H}\ord(\chi_v).\]
We have
$\Hom (\OO_v^\times/\cA_v, H)\cong\Hom (\F_v^\times/(\cA\bmod v) , H)\cong H[d_{\cA, H}(v)]$
as abelian groups, where $d_{\cA,H}(v)$ is as in Lemma \ref{lem:dhv_frobenian}. Hence,
\[s_{1,H}(v)+1=\sum_{h\in H[d_{\cA,H}(v)]}\ord h.\]
We now recall from the proof of Lemma \ref{lem:dhv_frobenian} that the function $d_{\cA,H}(v)$ is $S$-frobenian  with associated Galois group $\Gal(k_e/k)$, where $e=\exp(H)$, determined by the subsets  $\Sigma_d$. Using inclusion-exclusion, the mean of 
$s_{1,H}(v)+1$ equals
\begin{align*}
\frac{1}{[k_e:k]}\sum_{d\mid e}|\Sigma_d|\sum_{h\in H[d]}\ord h &= \frac{1}{[k_e:k]}\sum_{d\mid e}\sum_{c\mid\frac{e}{d}}\mu(c)[k_e:k_{cd}]\sum_{h\in H[d]}\ord h \\
&= \sum_{d\mid e}\sum_{c\mid \frac{e}{d}}\frac{\mu(c)}{[k_{cd}:k]}\sum_{h\in H[d]}\ord h\\
&=\sum_{f\mid e}\frac{1}{[k_f:k]}\sum_{d\mid f}\mu(f/d)\sum_{h\in H[d]}\ord h\\
&=\sum_{f\mid e}\frac{f|H_f|}{[k_f:k]},
\end{align*}
where $H_f$ denotes the set of elements of $H$ of order $f$. It is clear that \[\sum_{f\mid e}\frac{f|H_f|}{[k_f:k]}=\sum_{h\in H}\frac{\ord h}{[k_{\ord h}:k]},\] whence the claim.
\end{proof}

\begin{proposition}\label{prop:analytic_properties}
  Let $x\in \OO_S^\times\otimes \dual H$. The Fourier transform satisfies
  \begin{equation*}
    \widehat{f}_{\Lambda,H}(x;s)=\zeta_k(s)^{\varrho(k,H,\cA,x)}G(H,x;s),\quad \re s>1,
  \end{equation*}
  for a holomorphic function $G(s)=G(H,x;s)$ on the region
  \begin{equation*}
    \re s > 1-\frac{c}{\log(\lvert \Im s|+3)},
  \end{equation*}
  for some $c>0$, satisfying in this region the bound
  \begin{equation*}
    |G(s)|\ll (1+\lvert \Im s|)^{1/2}.
  \end{equation*}
  Moreover,
  \begin{equation*}
    \lim_{s\to 1}(s-1)^{\varrho(k,H,\cA,x)}\widehat{f}_{\Lambda,H}(x;s)=(\Res_{s=1}\zeta_k(s))^{\varrho(k,H,\cA,x)}\prod_{v\in\Omega_k}\frac{\widehat{f}_{\Lambda_v,H}(x_v;1)}{\zeta_{k,v}(1)^{\varrho(k,H,\cA,x)}},
  \end{equation*}
  where $\zeta_{k,v}(s)$ denotes the Euler factor of the Dedekind zeta function $\zeta_k(s)$ of $k$ at $v$ if $v$ is finite, and $\zeta_{k,v}(s)=1$ otherwise. When $x=1$ this limit is non-zero.
\end{proposition}

\begin{proof}
Recall the Euler product \eqref{eqn:Euler_product}.
  The same argument as in \cite[Lemma 3.6]{HNP2} shows that each single Euler factor $\widehat{f}_{\Lambda_v}(x_v;s)$ satisfies $\widehat{f}_{\Lambda_v}(x_v;s)\ll_{k,H}1$ on $\re(s)>0$ and $\widehat{f}_{\Lambda_v}(1;s)>0$ for $s\in\R$. We use these facts to control the finite product $\prod_{v\in S}\widehat{f}_{\Lambda_v}(x_v;s)$.
  
  The proposition is then an application of \cite[Prop 2.3]{HNP2} to the Euler product
  $\prod_{v\notin S}\widehat{f}_{\Lambda_v,H}(x_v;s)$,
  which satisfies the hypotheses of \cite[Prop 2.3]{HNP2} by Lemma \ref{lem:fourier_transform}, Proposition \ref{prop:s_S-frob} and the fact that $s_{x,H}(v)\leq |G|^2-1 < q_v$ by our assumption on $S$ in \S\ref{sec:Mob} (cf.~\cite[Prop 3.16]{HNP2}). 
 \end{proof}

\subsection{Proof of the asymptotic formula in Theorem \ref{thm:local_norms_weighted_by_ramification} }
Recall from  \eqref{eq:def_F_s} the Dirichlet series $F(s)$ relevant to Theorem \ref{thm:local_norms_weighted_by_ramification}.
Putting our results from \eqref{eqn:Mob} and \eqref{eq:poisson} together, we have
\begin{equation}\label{eq:F_s_expression}
  F(s)=\sum_{H\subset G}\frac{\mu(G/H)}{|\OO_k^\times\otimes \dual H|}\sum_{x\in\OO_S^\times\otimes\dual H}\widehat{f}_{\Lambda,H}(x;s),\quad\re s>1,
\end{equation}
and the analytic properties of the Fourier transforms
$\widehat{f}_{\Lambda,H}(x;s)$ were summarised in Proposition
\ref{prop:analytic_properties}.
For every proper subgroup $H\lneq G$, we see immediately from
\eqref{eq:def_varrho_k_G} that
\begin{equation}\label{eq:varrho_H_smaller_G}
\varrho(k,H,\cA)<\varrho(k,G,\cA).
\end{equation}
An application of the Selberg--Delange method \cite[Thm II.5.3]{Ten15} now gives the following (cf.~\cite[Prop 3.19]{HNP2}). 

\begin{proposition}\label{prop:counting_result}
  There exists $\delta=\delta(k,G,\cA)>0$ such that
  \begin{equation*}
    \sum_{\substack{\varphi\in\gextk \\ \Phi(\varphi)\leq B \\
		 \varphi_v \in \Lambda_v \forall v \in S}} \prod_v \mathfrak{e}_v(\varphi) = c_{k,G,\Lambda}B(\log B)^{\varrho(k,G,\cA)-1}+O(B(\log B)^{\varrho(k,G,\cA)-1-\delta}),
  \end{equation*}
  where
  \begin{equation*}
    c_{k,G,\Lambda}=\frac{1}{\Gamma(\varrho(k,G,\cA))|\OO_k^\times\otimes \dual G|}\sum_{\substack{x\in\OO_S^\times\otimes\dual G\\\varrho(k,G,\cA, x)=\varrho(k,G,\cA)}}\lim_{s\to 1}(s-1)^{\varrho(k,G,\cA)}\widehat{f}_{\Lambda,G}(x;s).
  \end{equation*}
\end{proposition}

Proposition \ref{prop:counting_result} shows the validity of the asymptotic formula in Theorem \ref{thm:local_norms_weighted_by_ramification}, so all that remains is to study the leading constant $c_{k,G,\Lambda}$ and prove its positivity under certain assumptions.

\subsection{The leading constant in Theorem \ref{thm:local_norms_weighted_by_ramification}	}
We first determine which $x$ contribute to the leading constant.  Recall that
for $x\in k^\times\otimes\dual G$, we denote by $x_v$ the image of $x$ under
$k^\times\otimes\dual G \to k_v^\times\otimes\dual G$, and by
$\cA_v\otimes\dual{G}$ the image of the map $\cA_v\otimes\dual{G}\to k_v^{\times}\otimes\dual{G}$.  
Recall from \cite[Lemma~3.20]{HNP2} the set $\mathcal{X} (k, G, \cA)$ defined as
$$\mathcal{X}(k,G,\mathcal{A}) = \{ x \in k^\times \otimes \dual{G} :
	x_v \in \mathcal{A}_v \otimes \dual{G} \text{ for all but finitely many } v\}.$$
	In \emph{loc.~cit.}~it is shown that $\mathcal{X}(k,G,\mathcal{A})$ is finite
	and 
	\begin{equation}\label{eq:X}
	\mathcal{X}(k,G,\mathcal{A}) = \{ x \in \OO_S^\times \otimes \dual{G} :
	x_v \in \mathcal{A}_v \otimes \dual{G} \text{ for all } v \notin S\}.
	\end{equation}

\begin{lemma}\label{lem:XkG}
  For $x\in\OO_S^\times\otimes \dual G$ we have
    \begin{equation*}
    \varrho(k,G,\cA,x) = \varrho(k,G,\cA) \Longleftrightarrow x\in\mathcal{X} (k, G, \cA).
  \end{equation*}
  Moreover, $\widehat{f}_{\Lambda_v,G}(x_v;1)=\widehat{f}_{\Lambda_v,G}(1;1)$ for $x \in \mathcal{X}(k, G, \cA)$ and $v\notin S$.
\end{lemma}

\begin{proof}
  Let $x\in\OO_S^\times\otimes \dual G$ and $v\notin S$. Recall that $s_{x,G}(v)\in\R$ by Proposition~\ref{prop:s_S-frob}. From the definition of $s_{x,G}(v)$ in
  Lemma~\ref{lem:fourier_transform}, we see that $s_{x,G}(v)\leq s_{1,G}(v)$,
  with equality if and only if $\langle\chi_v,x_v\rangle =1$ for all
  $\chi_v\in\Hom(\OO_v^\times/\cA_v,G)$. 
  The latter condition holds if and only if $x_v\in\cA_v\otimes \dual G$. 
  Hence, we see that
  \begin{equation}\label{eq:sigmas_equal}
    s_{x,G}(v) = s_{1,G}(v)\text{ for all }v\notin S \Longleftrightarrow x\in\mathcal{X} (k, G, \cA).
  \end{equation}
The equivalence in the lemma follows, as $s_{x,G}(v)\leq s_{1,G}(v)$ holds for
all $v\notin S$ and both functions are $S$-frobenian by Proposition
\ref{prop:s_S-frob}. 

The equality of Fourier transforms follows from Lemma~\ref{lem:fourier_transform} and \eqref{eq:X}.
\end{proof}

\begin{theorem} \label{thm:main_constant}
  Under the assumptions of Theorem \ref{thm:local_norms_weighted_by_ramification} and the additional assumptions on $S$ imposed at the start of \S \ref{sec:Mob}, the leading constant $c_{k,G,\Lambda}$ has the form
  \begin{align*}
    c_{k,G,\Lambda} = &\frac{(\Res_{s=1}\zeta_k(s))^{\varrho(k,G,\cA)}}{\Gamma(\varrho(k,G,\cA))|\OO_k^\times\otimes\dual G||G|^{|S_f|}}\prod_{v\notin S}\left(\sum_{\substack{\chi_v\in\Hom(\OO_v^\times,G)\\ \cA_v\subset\Ker\chi_v}}\frac{\ord\chi_v}{\Phi_v(\chi_v)}\right)\zeta_{k,v}(1)^{-\varrho(k,G,\cA)}\\
    &\times \sum_{\substack{\chi\in\Hom(\prod_{v\in S}k_v^\times,G)\\ \chi_v\in\Lambda_v \forall v\in S}}\left(\prod_{v\in S}\frac{\mathfrak{e}_v(\chi_v)}{\Phi_v(\chi_v)\zeta_{k,v}(1)^{\varrho(k,G,\cA)}}\right)\sum_{x\in\mathcal{X} (k, G, \cA)}\prod_{v\in S}\langle \chi_v,x_v\rangle,
  \end{align*}
where $S_f$ denotes the set of non-archimedean places in $S$, and the product over $v\notin S$ is non-zero. Moreover $c_{k,G,\Lambda} >0$ if there exists a sub-$G$-extension of $k$ which realises the given local conditions for all places $v$.
\end{theorem}

\begin{proof}
	The formula follows from writing the global Fourier transform as a product of local Fourier transforms, and applying Lemma~\ref{lem:fourier_transform}, Proposition \ref{prop:analytic_properties}, Proposition~\ref{prop:counting_result}, and Lemma \ref{lem:XkG} (cf.~\cite[Thm 3.22]{HNP2}). It remains to show that $c_{k,G,\Lambda}>0$ if there is a sub-$G$-extension $\varphi:\Gal(\bar{k}/k)\to G$ that satisfies $\varphi_v\in\Lambda_v$ for all $v\in S$. The argument is exactly the same as in \cite[\S 3.8]{HNP2} and hence omitted.
\end{proof}

Theorem \ref{thm:local_norms_weighted_by_ramification} follows immediately from
Proposition \ref{prop:counting_result} and Theorem \ref{thm:main_constant}.

\section{Distribution of genus numbers}
In this section we use Theorem \ref{thm:local_norms_weighted_by_ramification} to prove Theorem \ref{thm:genus}. We begin by studying some basic properties of genus numbers.

\subsection{Genus numbers}
For a Galois extension $K/k$ of number fields and places $v\in\Val_k$ and
$w_0\in \Val_K$ with $w_0\mid v$, we have the following equalities of subsets of $k_v^\times$:
\begin{equation*}
  \Norm_{K/k}\prod_{w\mid v}\OO_{K,w}^\times =
  \Norm_{K_{w_0}/k_v}\OO_{K,w_0}^\times = \OO_v^\times\cap
  \Norm_{K_{w_0}/k_v}K_{w_0}^\times = \OO_v^\times\cap \Norm_{K/k}\prod_{w\mid v}K_w^\times.
\end{equation*}
Hence, being the norm of an id\`ele of $K$ or of an \emph{integral} id\`ele of $K$ is
the same for integral id\`eles of $k$. The following result is a special case of the main theorem of \cite{Fur67}. It gives a purely adelic interpretation of the genus number. 

\begin{lemma}[{Furuta, \cite[\S5]{Fur67}}] \label{lem:genus_formula}
	Let $K/k$ be abelian. Then
	$$\mathfrak{g}_{K/k} =  \frac{h(k) \prod_{v \in \Val_k} \mathfrak{e}_v(K)}{[K:k][\O_k^\times : \O_k^\times \cap \Norm_{K/k} \prod_{w \in \Val_K} \O_{K,w}^\times ]}$$
	where $\mathfrak{e}_v(K)$ denotes the ramification index of a place of $K$ above $v$.
\end{lemma}

We have the following simple observations.

\begin{lemma} \label{lem:third_iso}
	Let $e$ be the exponent of $G$. Then we have 
	$$[\O_k^\times : \O_k^\times \cap \Norm_{K/k} \prod_{w \in \Val_K} \O_{K,w}^\times ] = 
	[\O_k^\times/\O_k^{\times e} : (\O_k^\times \cap \Norm_{K/k} \prod_{w \in \Val_K} \O_{K,w}^\times)/\O_k^{\times e} ].$$
\end{lemma}
\begin{proof}
	Third isomorphism theorem, since every element of $\O_k^{\times e}$
	is everywhere locally a norm by local class field theory, as observed in~\cite[Lemma~4.4]{HNP2}.
\end{proof}

\begin{lemma}\label{lem:subgroup_indicator} 
	Let $A$ be a finite group and $B \subset A$ a subgroup. Then
	$$\sum_{a \in A} \one_B(a) = |B| = |A|/[A:B].$$
      \end{lemma}
\begin{proof}
Immediate.
\end{proof}

We use these as follows. For a finitely generated subgroup $\mathcal{A} \subset
k^\times$, we let $f_\mathcal{A}$ be the indicator function for $G$-extensions
from which every element of $\mathcal{A}$ is everywhere locally a norm.
For cyclic subgroups $\langle\epsilon\rangle$, we
abbreviate this to $f_\epsilon=f_{\langle\epsilon\rangle}$. Then, writing $e$ for the exponent of $G$, Lemmas~\ref{lem:third_iso} and \ref{lem:subgroup_indicator} give
$$ \sum_{\epsilon \in \O_k^\times/\O_k^{\times e}} f_{\epsilon}(K)
= \frac{[\O_k^\times : \O_k^{\times e}]}{[\O_k^\times : \O_k^\times \cap \Norm_{K/k} \prod_{w \in \Val_K} \O_{K,w}^\times ]}.$$
(Note that $f_{\epsilon}$ is well defined since every element of $\O_k^{\times e}$ is everywhere locally a norm, see~\cite[Lemma~4.4]{HNP2}.)

Thus from Lemma \ref{lem:genus_formula} we obtain the following, which is the expression  we will
use to study the average value of the genus number. 

\begin{proposition} \label{prop:genus_formula}
	We have
	$$\sum_{\substack{\Gal(K/k)\cong G \\ \Phi(K/k) \leq B}} \mathfrak{g}_{K/k}  =
	\frac{h(k)}{[K:k][\O_k^\times : \O_k^{\times e}]}\sum_{\epsilon \in \O_k^\times/\O_k^{\times e}} 
	\sum_{\substack{\Gal(K/k)\cong G \\ \Phi(K/k) \leq B}} f_{\epsilon}(K)\prod_{v \in \Val_k} \mathfrak{e}_v(K).$$
\end{proposition}

\subsection{Proof of Theorem \ref{thm:genus}}
In this section we prove Theorem~\ref{thm:genus_constant}. This is a strengthening of Theorem \ref{thm:genus} that allows finitely many local conditions and also gives an explicit constant. 
 
\begin{theorem} \label{thm:genus_constant}
	Let $k$ be a number field and $G$ a finite abelian group with exponent $e$.
	Let $S$ be a finite set of places of $k$ satisfying the conditions 
	imposed at the start of \S \ref{sec:Mob}  for $\cA=\OO_k^\times$.
	For $v \in S$ let 
	$\Lambda_v$ be a  set of sub-$G$-extensions of $k_v$.
	Then there exist
	$ C_{k,G,\Lambda} \geq 0$ and $\delta=\delta(k,G) >0$ 
	such that
	\begin{equation*}
		\sum_{\substack{\varphi\in\gextk \\ \Phi(\varphi)\leq B \\
		 \varphi_v \in \Lambda_v \forall v \in S}} \mathfrak{g}_{K_\varphi/k}
		 = C_{k,G,\Lambda}B(\log B)^{\varrho(k,G)-1}
		+  O(B(\log B)^{\varrho(k,G)-1-\delta}), \quad B\to\infty.
	\end{equation*}
	The leading constant $C_{k,G,\Lambda}$ equals
	\begin{align*}
&\frac{h(k)(\Res_{s=1}\zeta_k(s))^{\varrho(k,G)}}{\Gamma(\varrho(k,G))|\OO_k^\times\otimes\dual G||G|^{|S_f|+1}[\O_k^\times : \O_k^{\times e}]}
\prod_{v\notin S}\left(\sum_{\chi_v\in\Hom(\OO_v^\times,G)}\frac{\ord\chi_v}{\Phi_v(\chi_v)}\right)\zeta_{k,v}(1)^{-\varrho(k,G)} \\
    & \times \sum_{\epsilon\in \O_k^\times/\O_k^{\times e}\cap\Sha_\omega(k,\mu_e)}\sum_{\substack{\chi\in\Hom(\prod_{v\in S}k_v^\times,G)\\ \chi_v\in\Lambda_v \forall v\in S\\ \epsilon_v \in\Ker\chi_v \forall v\in S}}\left(\prod_{v\in S}\frac{\mathfrak{e}_v(\chi_v)}{\Phi_v(\chi_v)\zeta_{k,v}(1)^{\varrho(k,G)}}\right)\sum_{x\in\mathcal{X} (k, G, 1)}\prod_{v\in S}\langle \chi_v,x_v\rangle,
    \end{align*}
	and $C_{k,G,\Lambda} >0$ if there exists a 
	sub-$G$-extension of $k$
	which realises the given local conditions for all places $v \in S$.
\end{theorem}

In this statement we use some notation concerning
Tate--Shafarevich groups, which we now introduce. For a number field $k$,
commutative group scheme $G$ over $k$, and a finite set of places $S$ of $k$, we write
\begin{align*}
	\Sha_S(k,G) = \Ker\left(\HH^1(k, G) \to \prod_{v \notin S} \HH^1(k_v,G)\right),
	\,\,
	\Sha_\omega(k,G) = \varinjlim_S \Sha_S(k,G).
\end{align*}
We are  interested  in the case $G= \mu_n$, where by Kummer theory we have
$$\Sha_\omega(k,\mu_n) = \{x\in k^\times/k^{\times n}\ :\ x_v\in k_v^{\times n} \text{ for all but finitely many }v\}.$$
The natural map $\O_k^\times/\O_k^{\times n}\to k^\times/k^{\times n}$ is injective and we write $\O_k^\times/\O_k^{\times n}\cap \Sha_\omega(k,\mu_n)$ to mean the intersection taken inside $k^\times/k^{\times n}$.

\begin{proof}[Proof of Theorem \ref{thm:genus_constant}]
As in Proposition \ref{prop:genus_formula} we have
	\begin{equation}\label{eq:genus_sum}
	\sum_{\substack{\varphi\in\gextk \\ \Phi(\varphi)\leq B \\
		 \varphi_v \in \Lambda_v \forall v \in S}} \mathfrak{g}_{K_\varphi/k} =
	\frac{h(k)}{|G|[\O_k^\times : \O_k^{\times e}]}\sum_{\epsilon \in \O_k^\times/\O_k^{\times e}} 
	\sum_{\substack{\varphi\in\gextk \\ \Phi(\varphi)\leq B \\
		 \varphi_v \in \Lambda_v \forall v \in S}} f_{\epsilon}(K_{\varphi})\prod_{v \in \Val_k} \mathfrak{e}_v(K_{\varphi}),
		 \end{equation}
		 where $K_\varphi$ denotes the fixed field of $\Ker\varphi$ in $\bar{k}$ and $e$ is the exponent of $G$.
		 
	We now prepare to apply Theorem~\ref{thm:local_norms_weighted_by_ramification}. For $\epsilon \in k^\times$, define $(\Lambda,\epsilon):=\left((\Lambda,\epsilon)_v\right)_{v\in\Omega_k}$ as follows. For $v\notin S$, let $(\Lambda,\epsilon)_v$ be the set of sub-$G$-extensions of $k_v$ corresponding to those extensions of local fields $L/k_v$ for which $\epsilon$ is a local norm from $L/k_v$. For $v\in S$, let $(\Lambda,\epsilon)_v$ be the subset of $\Lambda_v$ consisting of those extensions $L/k_v$ for which $\epsilon$ is a local norm from $L/k_v$. Note that this definition only depends on the image of $\epsilon$ in $k^\times/k^{\times e}$, and for $v\in S$ the set $(\Lambda,\epsilon)_v$ may be empty.

Applying Theorems~\ref{thm:local_norms_weighted_by_ramification} and~\ref{thm:main_constant} (with $\cA$ generated by $\epsilon$) to the inner sum of~\eqref{eq:genus_sum} yields 
\begin{align*}
\sum_{\substack{\varphi\in\gextk \\ \Phi(\varphi)\leq B \\
		 \varphi_v \in \Lambda_v \forall v \in S}} f_{\epsilon}(K_{\varphi})\prod_{v \in \Val_k} \mathfrak{e}_v(K_{\varphi}) = c_{k,G,(\Lambda,\epsilon)}B(\log B)^{\varrho(k,G,\epsilon)-1}
		+  O(B(\log B)^{\varrho(k,G,\epsilon)-1-\delta}), 
	\end{align*}
	where
        \begin{equation}\label{eq:def_varrho_eps}
          \varrho(k,G,\epsilon) = \sum_{g\in G\smallsetminus\{\id_G\}}\frac{\ord g}{[k_{\ord g,\epsilon}: k]}
          \quad \text{and }  k_{d,\epsilon}=k(\mu_d,\sqrt[d]{\epsilon})
        \end{equation}
and $c_{k,G,(\Lambda,\epsilon)}\geq 0$ is as in Theorem~\ref{thm:main_constant}, with local conditions coming from $\Lambda$ and the requirement that $\epsilon$ is everywhere locally a norm. If these two sets of conditions are not compatible, i.e.\ there are no $G$-extensions $\varphi$ with $\varphi_v\in\Lambda_v$ for all $v\in S$ from which a given $\epsilon$ is everywhere locally a norm, then $c_{k,G,(\Lambda,\epsilon)}=0$ and this $\epsilon$ does not contribute to the sum in~\eqref{eq:genus_sum}. Note that if $c_{k,G,(\Lambda,1)}=0$ then there are no $G$-extensions $\varphi$ with $\varphi_v\in\Lambda_v$ for all $v\in S$ and hence $c_{k,G,(\Lambda,\epsilon)}=0$ for all $\epsilon\in \O_k^\times/\O_k^{\times e}$. The next lemma shows that the main term in~\eqref{eq:genus_sum} comes from those $\epsilon$ lying in $\O_k^\times/\O_k^{\times e}\cap \Sha_\omega(k,\mu_e)$.

\begin{lemma}\label{lem:rho}
We have $\varrho(k,G,\epsilon) \leq \varrho(k,G) $, with equality if and only if $\epsilon\in \O_k^\times/\O_k^{\times e}\cap \Sha_\omega(k,\mu_e)$.
\end{lemma}

\begin{proof}
It follows from the definition \eqref{eq:def_varrho_eps} that $\varrho(k,G,\epsilon) \leq \varrho(k,G,1) $, with equality if and only if $\epsilon\in
k(\mu_d)^{\times d}$ for all $d\mid e$. This is equivalent to $\epsilon\in k_v^{\times e}$ for all but finitely many places $v$ of $k$ (see~\cite[Thm~1.6]{HNP2}, for example). 
\end{proof}

Thus, the main term in~\eqref{eq:genus_sum} is
	\begin{align*}
	&\frac{h(k)}{|G|[\O_k^\times : \O_k^{\times e}]}\sum_{\epsilon \in \O_k^\times/\O_k^{\times e}\cap  \Sha_\omega(k,\mu_e)} 
	\sum_{\substack{\varphi\in\gextk \\ \Phi(\varphi)\leq B\\
		 \varphi_v \in \Lambda_v \forall v \in S}} f_{\epsilon}(K_{\varphi})\prod_{v \in \Val_k} \mathfrak{e}_v(K_{\varphi}) = \\
		& 	\frac{h(k)}{|G|[\O_k^\times : \O_k^{\times e}]}\sum_{\epsilon \in \O_k^\times/\O_k^{\times e}\cap  \Sha_\omega(k,\mu_e)} c_{k,G,(\Lambda,\epsilon)}B(\log B)^{\varrho(k,G)-1}
		+  O(B(\log B)^{\varrho(k,G)-1-\delta}).
		 \end{align*} 
Hence in the setting of Theorem \ref{thm:genus_constant} we obtain the stated asymptotic formula with the leading constant
\begin{equation}\label{eq:C}
C_{k,G,\Lambda}=\frac{h(k)}{|G|[\O_k^\times : \O_k^{\times e}]}\sum_{\epsilon \in \O_k^\times/\O_k^{\times e}\cap  \Sha_\omega(k,\mu_e)} c_{k,G,(\Lambda,\epsilon)}.
\end{equation}
Applying Theorem~\ref{thm:main_constant} to the term for $\epsilon=1$ already shows that
$C_{k,G,\Lambda}$ is positive if there exists a sub-$G$-extension of $k$ which
realises the local conditions imposed by $\Lambda$ for all places $v \in S$. It
remains to prove that $C_{k,G,\Lambda}$ has the form claimed in
Theorem~\ref{thm:genus_constant}.

We begin by noting that the sum in~\eqref{eq:C} has either one or two terms. To see this, let $2^r$ be the largest power of $2$ dividing $e$. It follows from~\cite[Theorem~9.1.11]{NSW08} that $\Sha_\omega(k,\mu_e)$ is trivial unless $k(\mu_{2^r})/k$ is non-cyclic, in which case $\Sha_\omega(k,\mu_e)\cong \Z/2\Z$ (see~\cite[Lemma~4.9]{HNP2}).

We next show that the elements $\epsilon\in \O_k^\times/\O_k^{\times e}\cap\Sha_\omega(k,\mu_e)$ can only be non-trivial at a uniformly bounded (in terms of $[k:\bbQ]$) subset of the places in $S$.
We write $\epsilon_v$ for the image of $\epsilon$ in $k_v^\times/k_v^{\times e}$. Recall that a place $v$ of a number field $L$ is said to \emph{split} (or \emph{decompose}) in an extension $M/L$ if there exist at least two distinct places of $M$ above $v$.

\begin{lemma}\label{lem:f_e} Let $\epsilon\in \O_k^\times/\O_k^{\times e}\cap\Sha_\omega(k,\mu_e)$, where $e$ is the exponent of $G$. Let $2^r$ be the largest power of $2$ dividing $e$ and let $R$ be a set of places of $k$ containing all places above $2$ that do not split in $k(\mu_{2^r})/k$. Then $\epsilon_v\in k_v^{\times e}$ for all $v\notin R$.
\end{lemma}

\begin{proof}
By definition, there exists a cofinite set of places $T$ such that $\epsilon_v\in k_v^{\times e}$ for all $v\in T$. Let $U=\Omega_k\setminus R$. By Grunwald--Wang~\cite[Theorem~9.1.11]{NSW08} we have
\[\Ker \left(k^\times/k^{\times e}\to \prod_{v\in T}k_v^\times/k_v^{\times e}\right)=\Ker \left(k^\times/k^{\times e}\to \prod_{v\in T\cup U}k_v^\times/k_v^{\times e}\right),\]
proving that $\epsilon_v\in k_v^{\times e}$ for all $v\notin R$.
\end{proof}

By Lemma \ref{lem:f_e} and \eqref{eq:X}, we have $\mathcal{X}(k, G,
\epsilon)=\mathcal{X}(k, G, 1)$ for $\epsilon \in
\Sha_\omega(k,\mu_e)$. Together with Theorem~\ref{thm:main_constant},
Lemma~\ref{lem:rho}, Lemma~\ref{lem:f_e} and our assumptions on $S$, this shows
that, for $\epsilon \in \O_k^\times/\O_k^{\times e}\cap  \Sha_\omega(k,\mu_e)$,
\begin{align*}
 c_{k,G,(\Lambda,\epsilon)}=&\frac{(\Res_{s=1}\zeta_k(s))^{\varrho(k,G)}}{\Gamma(\varrho(k,G))|\OO_k^\times\otimes\dual G||G|^{|S_f|}}\prod_{v\notin S}\left(\sum_{\chi_v\in\Hom(\OO_v^\times,G)}\frac{\ord\chi_v}{\Phi_v(\chi_v)}\right)\zeta_{k,v}(1)^{-\varrho(k,G)}\\
    &\times \sum_{\substack{\chi\in\Hom(\prod_{v\in S}k_v^\times,G)\\ \chi_v\in\Lambda_v \forall v\in S\\ \epsilon_v \in\Ker\chi_v \forall v\in S}}\left(\prod_{v\in S}\frac{\mathfrak{e}_v(\chi_v)}{\Phi_v(\chi_v)\zeta_{k,v}(1)^{\varrho(k,G)}}\right)\sum_{x\in\mathcal{X} (k, G, 1)}\prod_{v\in S}\langle \chi_v,x_v\rangle.
    \end{align*}
Plugging this into \eqref{eq:C} gives the correct leading constant for Theorem \ref{thm:genus_constant}.
\end{proof}


\subsection{Remarks on Theorem~\ref{thm:genus_constant}}
The constant $c$ in Theorem~\ref{thm:genus} is equal to $C_{k,G,\Lambda}/\lvert \Aut(G)|$, where $C_{k,G,\Lambda}$ is as in Theorem \ref{thm:genus_constant} and $\Lambda_v$ is taken to be the set of all sub-$G$-extensions of $k_v$ for each $v\in S$. This is because in the counting function in Theorem~\ref{thm:genus} we do not fix a choice of isomorphism from $\Gal(K/k)$ to $G$.
Next, we observe that the exponent occurring in Theorem~\ref{thm:genus} is an integer.
\begin{lemma} \label{lem:varrho_integer}
The number
$\varrho(k,G)$ is a non-negative integer.
\end{lemma}

\begin{proof}
Let $\varphi(n)=|(\bbZ/n\bbZ)^\times|$.  We have
$$
\varrho(k,G)=\sum_{n>1}\frac{n}{[k(\mu_n):k]}\#\{g\in G : \ord g=n\}.$$
It suffices to note that $[k(\mu_n):k]\mid \varphi(n)$ and that
$\varphi(n) \mid \#\{g\in G : \ord g=n\}$, since if $g \in G$ has order $n$ then
so does $g^a$ for all $1 \leq a \leq n$ with $\gcd(a,n) = 1$, and these elements
are distinct.
\end{proof}

Next we show that the group $\O_k^\times/\O_k^{\times
  e}\cap\Sha_\omega(k,\mu_e)$, which occurs in the formula for the constant $C_{k,G,\Lambda}$ in Theorem~\ref{thm:genus_constant}, can be non-trivial.

\begin{example} \label{ex:unit_local_norm}
	Here we give an example of a number field $k$ such that the map
	$$\O_k^{\times}/\O_k^{\times 8} \to 
	\prod_{v \in \Val_k} \O_{v}^{\times}/\O_{v}^{\times 8} $$
	is not injective and hence $\O_k^\times/\O_k^{\times e}\cap\Sha_\omega(k,\mu_e)\cong \Z/2$.
	
	Take $k = \Q(\sqrt{7})$. Then it is well known that $16$ is an $8$th
	power everywhere locally in $k$, but not globally an $8$th power
	(Wang's counter-example to  Grunwald's theorem).
	Note that $2$ is ramified in $k$, so that $(2) = \fp^2$ for some prime
	ideal $\fp$. Then $\fp^8/(16)$ is the trivial ideal.
	But the class number of $k$ is $1$ so, writing $\fp=(a)$, we have
	that $a^8/16$ is a unit that is an $8$th power everywhere locally but not an $8$th power globally, as required.
	
	Explicitly, we have $\fp = (3 + \sqrt{7})$, so we may take
	$a = \sqrt{7} + 3$. Then $u = a^8/16 = 32257 + 12192\sqrt{7}$ is the sought-after unit. Note that
	$u = \epsilon^4$, where $\epsilon = 8 + 3\sqrt{7}$ is the fundamental unit.
\end{example}

\section{Zero density with fixed genus number}
In this section, we prove Theorem \ref{thm:zero_percent}. Let $e$ again denote
the exponent of $G$. From Lemma
\ref{lem:genus_formula}, we see at once that any $G$-extension $K/k$ satisfies
\begin{equation*}
  \mathfrak{g}_{K/k}\geq \frac{h(k)}{|G|[\O_k^\times:\O_k^{\times
      e}]}\prod_{v\in\Val_k}\mathfrak{e}_v(K)\gg 2^{\omega(K/k)},
\end{equation*}
where $\omega(K/k)$ denotes the number of places of $k$ that ramify in
$K/k$. With this observation, Theorem \ref{thm:zero_percent} follows
immediately from the following result.

\begin{theorem}
  Let $k$ be a number field, $G$ a finite abelian group and $r\in\N$. Then
   \begin{equation}\label{eq:zero_density_ramified_primes}
    \lim_{B\to\infty}\frac{|\{K/k\ :\ \Gal(K/k)\cong G,\ \Phi(K/k)\leq
      B\ \text{ and }\ \omega(K/k)\leq r\}|}{|\{K/k\ :\ \Gal(K/k)\cong G,\ \Phi(K/k)\leq
      B\}|}=0.
  \end{equation}
\end{theorem}

\begin{proof}
  We show that the limit in \eqref{eq:zero_density_ramified_primes} is smaller
  than any $\varepsilon>0$. To this end, we will apply Chebyshev's inequality in the form of
  \cite[Lemma 5.1]{FW17}, which is a straightforward generalisation of its
  special case \cite[Lemma 3.1]{EPW17} for $k=\Q$.

  Let $z$ be a fixed positive constant, chosen sufficiently large in terms of
  $k,G$ and $\varepsilon$. Write $P(z)=\{\fp\ :\ N(\fp)\leq z\}$ for the set of
  all non-zero prime ideals of $\O_k$ with norm bounded by $z$. Moreover, fix a
  sufficiently large finite set $S_0$ of places of $k$.

  For any $B>1$, let
  \begin{equation*}
    \mathscr{A}=\{K/k\ :\ \Gal(K/k)\cong G\text{ and }\Phi(K/k)\leq B\}
  \end{equation*}
  and write $N=|\mathscr{A}|$. For any prime ideal $\fp\in S_0$, we take
  $\mathscr{A}_\fp=\mathscr{A}$. For $\fp\notin S_0$, we take
  \begin{equation*}
    \mathscr{A}_{\fp}= \{K\in\mathscr{A}\ :\ \fp \text{ ramifies in }K\}.
  \end{equation*}
  For any distinct prime ideals $\fp,\fq$, write
  $\mathscr{A}_{\fp,\fq}=\mathscr{A}_\fp\cap\mathscr{A}_\fq$. To determine the
  cardinalities of the sets $\mathscr{A}_\fp$ and $\mathscr{A}_{\fp,\fq}$
  asymptotically for $B\to\infty$, we apply \cite[Theorem 3.1]{HNP2}.
  We take $\mathcal{A}$ in the statement of \cite[Theorem 3.1]{HNP2} to be the
  trivial subgroup. For $v\in S_0$, we take
  $\Lambda_v=\Hom(k_v^\times,G)$. For $v\in S\smallsetminus S_0$, we take
  \begin{equation*}
    \Lambda_v=\{\chi\in\Hom(k_v^\times,G)\ :\ \chi\text{ is ramified}\}.
  \end{equation*}
  As $S_0$ is sufficiently large, the explicit description of the
  leading constant in \cite[Theorem 3.22]{HNP2} applies to any finite set of
  places $S\supset S_0$.
  Hence, applying \cite[Theorem 3.1 and Theorem 3.22]{HNP2} with $S=S_0$,
  $S=S_0\cup\{\fp\}$ and $S=S_0\cup\{\fp,\fq\}$, respectively, we obtain the estimates
  \begin{align*}
    N = |\mathscr{A}| & = c_{k,G}B(\log B)^{\omega(k,G)-1}+O(B(\log
    B)^{\omega(k,G)-1-\alpha}),\\
    |\mathscr{A}_\fp| &= \delta_\fp N + R_{\fp},\\
    |\mathscr{A}_{\fp,\fq}|& = \delta_\fp\delta_\fq N + R_{\fp,\fq},
  \end{align*}
  where $\omega(k,G)>0$ is given in Remark
  \ref{rmk:average}, $\alpha=\alpha(k,G)>0$, $c_{k,G}>0$,
  \begin{equation*}
    R_{\fp} = O_\fp(B(\log B)^{\omega(k,G)-1-\alpha}),\quad R_{\fp,\fq} = O_{\fp,\fq}(B(\log B)^{\omega(k,G)-1-\alpha}),
  \end{equation*}
  and $\delta_\fp$ is given as follows. For $\fp\in S_0$ we have $\delta_\fp=1$, and
  for $\fp\notin S_0$ corresponding to the place $v$, we have
  \begin{equation*}
    \delta_\fp =
    \frac{\sum_{\chi_v\in\Lambda_v}\frac{1}{\Phi_v(\chi_v)}}{\sum_{\chi_v\in\Hom(k_v^\times,G)}\frac{1}{\Phi_v(\chi_v)}}.    
  \end{equation*}
  Here we have used the fact that Euler factors for places $v\in
  S\smallsetminus S_0$ can be split off the formula for the constant
  $c_{k,G,\Lambda}$ given in \cite[Theorem 3.22]{HNP2}, as explained in the
  proof of \cite[Lemma 4.5]{HNP2}. We evaluate $\delta_\fp$
  further as follows. By \cite[Lemma 3.10]{HNP2}, we have, for $v\notin S_0$,
  \begin{equation*}
   \frac{1}{|G|} \sum_{\chi_v\in\Hom(k_v^\times,G)}\frac{1}{\Phi_v(\chi_v)} =
    1+(\lvert\Hom(\F_v^\times,G)|-1)q_v^{-1} = 1+(|G[q_v-1]|-1)q_v^{-1}. 
  \end{equation*}
  As the unramified homomorphisms $\chi_v$ contribute exactly the ``$1$''
  in the above formula, we also get  that
  \begin{equation*}
    \frac{1}{|G|}\sum_{\chi_v\in\Lambda_v}\frac{1}{\Phi_v(\chi_v)} =
    (\lvert \Hom(\F_v^\times,G)|-1)q_v^{-1} = (|G[q_v-1]|-1)q_v^{-1}, 
  \end{equation*}
  and hence 
  \begin{equation*}
    \delta_\fp=\frac{(|G[q_v-1]|-1)q_v^{-1}}{1+(|G[q_v-1]|-1)q_v^{-1}}\geq\frac{G[q_v-1]-1}{2q_v}\geq\frac{1}{2q_v}  
  \end{equation*}
  for sufficiently large $v$ with $\gcd(q_v-1,e)>1$. As the set of $v$ satisfying
  the latter condition has positive density by Chebotarev, we have shown that
  the series
  \begin{equation*}
    \sum_\fp\delta_\fp
  \end{equation*}
  diverges. For any field $K\in\mathscr{A}$, we clearly have 
   \begin{equation*}
    \omega_z(K/k):=|\{\fp\in P(z)\ :\ K\in \mathscr{A}_\fp\}|\leq \omega(K/k),
  \end{equation*}
  hence we are interested in bounding 
  $$E(B;z,r) = \{ K\in\mathscr{A} : \omega_z(K/k)\leq r\}$$
 for sufficiently large $B$. 
 To do so, we consider the mean
  \begin{align*}
    M(z)&:=\frac{1}{N}\sum_{K\in\mathscr{A}}\omega_z(K/k) =
          \frac{1}{N}\sum_{\fp\in P(z)}|\mathscr{A}_\fp|\\
    &=\sum_{\fp\in P(z)}\delta_\fp+O_z((\log B)^{-\alpha}) = U(z) + O_z((\log B)^{-\alpha}),
  \end{align*}
  where we have written $U(z)=\sum_{\fp\in P(z)}\delta_\fp$. As the sum over
  $\delta_\fp$ diverges, we may choose $z$ large enough so that $U(z)>4r$. Then we
  have $M(z)>U(z)/2>2r$ for all sufficiently large $B$. By \cite[Lemma
  5.1]{FW17}, we get
  \begin{equation*}
    \frac{E(B;z,r)}{N}\leq\frac{4}{M(z)^2}\left(U(z)+\frac{1}{N}\sum_{\fp,\fq\in
        P(z)}|R_{\fp,\fq}|+\frac{2U(z)}{N}\sum_{\fp\in
        P(z)}|R_\fp|+\left(\frac{1}{N}\sum_{\fp\in P(z)}|R_\fp|\right)^2\right),
  \end{equation*}
  where $R_{\fp,\fp}$ is interpreted as $R_\fp$. The expression on the
  left-hand side of \eqref{eq:zero_density_ramified_primes} is bounded above by
  \begin{equation*}
    \lim_{B\to\infty}\frac{E(B;z,r)}{N} \leq
    \frac{16}{U(z)^2}\big(U(z)+0+0+0\big) = \frac{16}{U(z)}<\varepsilon,
  \end{equation*}
  provided  $z$ is chosen large enough. 
\end{proof}

\section{Narrow genus numbers and generalisations}
We finish by discussing the narrow genus number $\mathfrak{g}_{K/k}^+$. For abelian $K/k$, this is defined analogously to Definition \ref{def:genus}, but instead taking the largest extension of $K$ which is abelian over $k$ and unramified only at all non-archimedean places of $K$. Warning: some authors take this as the definition of the genus number; this is due to Fr\"{o}hlich's original convention \cite{Fro59} and is, for example, the definition  used in \cite{Kim20}. As $\mathfrak{g}_{K/k}^+\geq \mathfrak{g}_{K/k}$, it is clear that an analogue of Theorem \ref{thm:zero_percent} holds for narrow genus numbers as well.

Analogously to Lemma \ref{lem:genus_formula}, we have the formula (see \cite[Thm.~4]{Gur77})
\begin{equation} \label{eqn:narrow_genus}
	\mathfrak{g}_{K/k}^+ =  \frac{h^+(k) \prod_{v \in \Val_k^{\textrm{f}}} \mathfrak{e}_v(K)}{[K:k][\O_k^{\times,+} : \O_k^{\times,+} \cap \Norm_{K/k} \prod_{w \in \Val_K} \O_{K,w}^\times ]}
\end{equation}
where $h^+(k)$ denotes the narrow class number of $k$, $\O_k^{\times,+}$ denotes the group of totally positive units of $k$, and $\Val_k^{\textrm{f}}$ denotes the set of finite places of $k$. 

The narrow genus number and the genus number only differ by a power of $2$. More precisely, $\mathfrak{g}_{K/k}^+/\mathfrak{g}_{K/k}$ divides $2^{r_1}$ where $r_1$ is the number of real places of $k$. This can be seen by taking the quotient of \eqref{eqn:narrow_genus} by the expression in Lemma~\ref{lem:genus_formula} and recalling that $h^+(k)/h(k)=2^{r_1}/[\O_k^\times:\O_k^{\times,+}]$ (see~\cite[Exercise 3]{CF}, for example).
Moreover, for odd degree extensions of $\Q$ the genus number and the narrow genus number coincide,
see \cite[2.9 Proposition]{Fro83}. 

Using \eqref{eqn:narrow_genus} in place of Lemma \ref{lem:genus_formula} during the proof of Theorem \ref{thm:genus_constant}, one can obtain an asymptotic formula for the sum of $\mathfrak{g}_{K/k}^+$  with the same order of magnitude, but  a different leading constant.

More generally, Horie \cite{Horie83} has defined the genus field with respect
to an arbitrary modulus of $K$ and given a formula for the corresponding
generalisation of the genus number, see~\cite[Corollary of
Theorem~2]{Horie83}. For moduli induced from a fixed modulus of $k$, our
methods can also be used to give an asymptotic formula for the sum of these
generalised genus numbers, again having the same order of magnitude as in
Theorem~\ref{thm:genus_constant} but with a different leading constant.


\begin{thebibliography}{xx}
\bibitem{CF} 
J.W.S. Cassels, A. Fr\"{o}hlich. 
\emph{Algebraic Number Theory.} 
Second Edition. London Mathematical Society. 2010.

\bibitem{CL84}
H. Cohen, H. W. Lenstra, Jr., Heuristics on class groups of number fields,
\textit{Number theory, Noordwijkerhout 1983 (Noordwijkerhout, 1983)}, 33--62, Lecture Notes in Math., \textbf{1068}, Springer, Berlin, 1984.

\bibitem{DW88}
B. Datskovsky, D. J. Wright, Density of discriminants of cubic extensions. \textit{J. reine angew. Math.} \textbf{386} (1988), 116--138.

\bibitem{DH71}
H. Davenport, H. Heilbronn, On the density of discriminants of cubic fields. II. \textit{Proc. Roy. Soc. London Ser. A} \textbf{322} (1971), no. 1551, 405--420.

\bibitem{EPW17}
{J. Ellenberg, L. B. Pierce, M. M. Wood},
{On $\ell$-torsion in class groups of number fields}.
\emph{Algebra Number Theory} {\bf11} (2017), No. 8, 1739--1778.

\bibitem{HNP}
{C. Frei, D. Loughran, R. Newton}, 
{The Hasse norm principle for abelian extensions}. 
\emph{Amer. J. Math.} {\bf140} (2018), no. 6, 1639--1685.


\bibitem{HNP2}
{C. Frei, D. Loughran, R. Newton}, appendix by Y. Harpaz and O. Wittenberg,
Number fields with prescribed norms. 
\emph{Comment. Math. Helv.} {\bf97} (2022), 133--181.

\bibitem{FW17}
{C. Frei, M. Widmer},
{Average bounds for the $\ell$-torsion in class groups of cyclic extensions}.
\emph{Res. Number Theory} {\bf4}, 34 (2018). \texttt{https://doi.org/10.1007/s40993-018-0127-9}


\bibitem{Fro83}
{A. Fr\"{o}hlich},
{Central extensions, {G}alois groups, and ideal class groups of number fields},
{\emph{Contemporary Mathematics}, Vol. 24 (1983)},
{American Mathematical Society, Providence, RI}.

\bibitem{Fro59}
A. Fr\"{o}hlich, The genus field and genus group in finite number fields. I,
II. \textit{Mathematika} {\bf 6} (1959), 40--46, 142--146.

\bibitem{Fur67}
{Y. Furuta},
{The genus field and genus number in algebraic number fields}.
\emph{Nagoya Math. J.} {\bf29} (1967), 281--285. 

\bibitem{Gur77}
{S. J. Gurak}, Ideal-theoretic characterization of the relative genus field. \emph{J. reine angew. Math.} {\bf296} (1977), 119--124.

\bibitem{Horie83}
{M. Horie}, On the genus field in algebraic number fields. \emph{Tokyo J. Math.} {\bf 6}(2) (1983), 363--380.

\bibitem{Kim20}
{H. H. Kim}, Genus numbers of cyclic and dihedral extensions of prime degree. 
\emph{Acta Arith.} {\bf192}(3) (2020), 289--300.

\bibitem{MTT20}
K. J. McGown, F. Thorne, A. Tucker, 
Counting quintic fields with genus number one. 
\emph{Math. Res. Lett.}, to appear.

\bibitem{MT16}
K. J. McGown, A. Tucker, 
Statistics of genus numbers of cubic fields. 
\emph{Ann. Inst. Fourier}, Online first, 19 p.

\bibitem{NSW08}
{J. Neukirch, A. Schmidt, K. Wingberg}, 
\textit{Cohomology of Number Fields}.
Second edition. Grundlehren der Mathematischen Wissenschaften {\bf323}, Springer-Verlag, 2008.

\bibitem{Ser12}
{J.-P. Serre}, \emph{Lectures on $N_X(p)$}.
Chapman \& Hall/CRC Research Notes in Mathematics, 11. CRC Press, Boca Raton, FL, 2012.

\bibitem{Ten15}
{G. Tenenbaum}, \emph{Introduction to analytic and probabilistic number theory}. Third edition. Graduate Studies in Mathematics, 163. American Mathematical Society, Providence, RI, 2015. 

\bibitem{Woo10} 
{M.~M.~Wood}, 
{On the probabilities of local behaviors in abelian field extensions}.
{\em Compositio Math.} {\bf 146} (2010), no.~1, 102--128.
\end{thebibliography}
\end{document}